\documentclass[11pt, reqno]{amsart}

\usepackage{amsmath, amsthm, amssymb, amsfonts}
\usepackage{eucal}
\usepackage{hyperref}

\usepackage[top=3.75cm, bottom=3cm, left=3cm, right=3cm]{geometry}
\theoremstyle{plain}
\newtheorem{theorem}{Theorem}
\newtheorem{question}[theorem]{Question}
\newtheorem{lemma}[theorem]{Lemma}

\theoremstyle{definition}
\newtheorem{remark}[theorem]{Remark}

\DeclareMathOperator{\Proj}{Proj}
\DeclareMathOperator{\Sym}{Sym}

\newcommand{\cE}{\mathcal{E}}
\newcommand{\cO}{\mathcal{O}}
\newcommand{\cX}{\CMcal{X}}

\newcommand{\bA}{\mathbf{A}}
\newcommand{\bC}{\mathbf{C}}
\newcommand{\bP}{\mathbf{P}}
\newcommand{\bQ}{\mathbf{Q}}

\newcommand{\rH}{\mathrm{H}}

\newcommand{\tX}{\widetilde{X}}

\begin{document}

\date{\today}

\title[Rationality does not specialize among terminal fourfolds]{Rationality does not specialize among \\ terminal fourfolds}

\author{Alexander Perry}
\address{Department of Mathematics, Columbia University, New York, NY 10027 \smallskip}
\email{aperry@math.columbia.edu}

\thanks{This work was partially supported by an NSF postdoctoral fellowship, DMS-1606460.}

\begin{abstract}
We show that rationality does not specialize in flat projective 
families of complex fourfolds with terminal singularities. 
This answers a question of Totaro, who established the 
analogous result in all dimensions greater than 4. 
\end{abstract}

\maketitle


\section{Introduction}

Rationality behaves subtly in families of complex algebraic varieties. 
In general, given a flat projective family, the locus of rational fibers forms 
a countable union of locally closed subsets of the base~\cite[Proposition~2.3]{defernex}. 
Recently, Hassett, Pirutka, and Tschinkel~\cite{HPT} produced a smooth 
projective family of fourfolds where none of these locally closed subsets 
is dense, but their union is. In particular, rationality is neither an open nor 
closed condition in smooth families. 

This paper concerns the question of whether the locally 
closed subsets parameterizing the rational fibers of a family 
are actually closed, i.e. whether rationality specializes. 

\begin{question}
\label{question-specialization}
Given a flat projective family of complex varieties, does geometric rationality 
of the generic fiber imply the same of every fiber? 
\end{question}

Without further restrictions, the answer is negative, as shown by a family 
of smooth cubic surfaces degenerating to a cone over a smooth cubic curve. 
However, if the fibers of the family are required to be smooth of dimension 
at most $3$, de Fernex and Fusi~\cite{defernex} proved the answer is positive. 
In fact, as Totaro observed, it follows from the results of~\cite{defernex} and Hacon and 
M\textsuperscript{c}Kernan~\cite{hacon-mckernan} that the answer remains 
positive if the fibers are allowed to have log terminal singularities and dimension at most~$3$.

In higher dimensions, however, Totaro~\cite{totaro-specialization} showed that 
rationality does not specialize among varieties with mild singularities. 
Namely, specialization fails in every dimension greater than~$4$ if terminal singularities 
(the mildest type of singularity arising from the minimal model program) are allowed, 
and in dimension $4$ if canonical singularities (the second mildest type of singularity) 
are allowed. This left open the possibility that rationality specializes among 
terminal fourfolds. The purpose of this paper is to show that this fails too: 

\begin{theorem}
\label{main-theorem}
There is a flat projective family of fourfolds over a Zariski open neighborhood 
$U$ of the origin $0 \in \bA^1$ in the complex affine line, such that: 
\begin{enumerate}
\item \label{c1} All the fibers have terminal singularities. 
\item \label{c2} The fibers over $U \setminus \{0\}$ are rational. 
\item \label{c3} The fiber over $0$ is stably irrational. 
\end{enumerate}
\end{theorem}

\begin{remark}
By taking products with projective spaces, the theorem gives families 
of varieties of any dimension at least $4$ satisfying conditions~\eqref{c1}-\eqref{c3}, 
and hence reproves Totaro's result. 
\end{remark}

Our proof of Theorem~\ref{main-theorem} closely follows~\cite{totaro-specialization}. 
There, starting from a stably irrational smooth quartic fourfold $Y \subset \bP^5$ 
(known to exist by~\cite{totaro-hypersurfaces}), Totaro constructs a family of fivefolds 
satisfying conditions~\eqref{c1}-\eqref{c3} in Theorem~\ref{main-theorem} by 
deforming the cone over $Y$ to rational fivefolds. In fact, starting from any stably 
irrational smooth hypersurface $Y \subset \bP^n$ which is Fano of index at least $2$, 
his argument can be run to produce a family of $n$-folds satisfying~\eqref{c1}-\eqref{c3}. 
It is thus tempting to prove Theorem~\ref{main-theorem} by taking such a $Y \subset \bP^4$. 
However, then the only potential candidate for $Y$ is a stably irrational cubic threefold, 
the existence of which is a difficult open problem. 
Our idea is to instead take $Y$ to be a quartic double solid. Then $Y$ is a Fano threefold 
of index~$2$, and can be chosen to be stably irrational by Voisin's seminal work~\cite{voisin}. 
Although $Y$ is not a hypersurface in projective space, it is a hypersurface in a 
\emph{weighted} projective space, which we show is enough to run Totaro's argument. 

The natural question left open by this paper is whether rationality specializes 
among smooth varieties of dimension greater than $3$. 

\subsection*{Conventions} 
We work over the field of complex numbers $\bC$. 
For positive integers $a_0, \dots, a_n$, we denote by $\bP(a_0, \dots, a_n)$ 
the weighted projective space with weights $a_i$. 
We use superscripts to denote that a weight is repeated with multiplicity, 
e.g. $\bP(1^4,2) = \bP(1,1,1,1,2)$. 
For a vector bundle $\cE$ on a scheme $S$, the associated projective 
bundle is $\bP(\cE) = \Proj_S(\Sym(\cE^{\vee}))$. 

\subsection*{Acknowledgements}
Theorem~\ref{main-theorem} was conceived during Burt Totaro's talk on~\cite{totaro-specialization} 
at the Higher Dimensional Algebraic Geometry Conference at the University of Utah in July 2016. 
I thank Burt for his comments on a draft of this paper, and the organizers of the conference 
for creating a stimulating environment. 


\section{Proof of Theorem~\ref{main-theorem}}
Let $Y \to \bP^3$ be a quartic double solid, i.e. a double cover of $\bP^3$ 
branched along a smooth quartic surface. 
We regard $Y$ as a hypersurface in the weighted projective space $\bP(1^4, 2)$, 
cut out by a polynomial of the form 
\begin{equation*}
f_4(x_0, \dots, x_4) = x_4^2 - h_4(x_0, \dots, x_3),
\end{equation*}
where $h_4(x_0, \dots, x_3)$ is a quartic. 
Let $X \subset \bP(1^4, 2, 1)$ be the cone over $Y$ defined by the same polynomial 
$f_4(x_0, \dots, x_4)$ in the bigger weighted projective space $\bP(1^4, 2, 1)$.
For a stably irrational choice of $Y$, the variety $X$ will form the central fiber 
in the promised family of fourfolds. 

\begin{lemma}
\label{lemma-X}
The following hold:
\begin{enumerate}
\item \label{lemma-X-1} $X$ is birational to $Y \times \bP^1$. 
\item \label{lemma-X-2} $X$ has terminal singularities. 
\end{enumerate}
\end{lemma}

\begin{proof}
Let $H$ denote the pullback of the hyperplane class on $\bP^3$ to $Y$. 
Define 
\begin{equation*}
\pi \colon \tX = \bP(\cO_Y(-H) \oplus \cO_Y) \to Y. 
\end{equation*}
There is a natural morphism $\tX \to \bP(1^4, 2,1)$ given as follows.  
Let $\zeta$ denote the divisor corresponding to the relative $\cO(1)$ line bundle on $\tX$. 
Then 
\begin{align*}
\pi_*(\cO_{\tX}(\zeta)) & = \cO_Y(H) \oplus \cO_Y ,  \\ 
\pi_*(\cO_{\tX}(2\zeta)) & = \cO_Y(2H) \oplus \cO_Y(H) \oplus \cO_Y. 
\end{align*}
Hence $\rH^0(\tX, \cO_{\tX}(\zeta)) \cong \bC^4 \oplus \bC$, and 
$\rH^0(\tX, \cO_{\tX}(2\zeta))$ has a canonical $1$-dimensional subspace 
corresponding to the canonical section of $\cO_Y(2H)$. 
This data specifies the morphism \mbox{$\tX \to \bP(1^4, 2, 1)$}. 
In fact, this morphism factors through $X \subset \bP(1^4, 2,1)$ and gives a resolution 
of singularities $f \colon \tX \to X$ with a single exceptional divisor 
\begin{equation*}
E = \bP(\cO_Y) \subset \tX, 
\end{equation*} 
which is contracted to $[0,0,0,0,0,1] \in X$. In particular, \eqref{lemma-X-1} holds. 

Note that $X$ is normal with $\bQ$-Cartier canonical divisor. 
We show that the discrepancy of the exceptional divisor $E$ above is $1$, 
so that~\eqref{lemma-X-2} holds. 
Write $K_{\tX} = f^*(K_X) + aE$. Then by adjunction 
\begin{equation*}
K_E = (K_{\tX} + E)|_E = (a+1)E|_E. 
\end{equation*}
Observe that $E \cong Y$, so $K_E = -2H$, and $E = \zeta - \pi^*H$, so $E|_E = -H$. 
We conclude $a = 1$. 
\end{proof}

Next choose a nonzero polynomial $g_3(x_0, \dots, x_4) \in \rH^0(\bP(1^4,2), \cO(3))$ 
of weighted degree~$3$. 
We consider the flat family $\cX \to \bA^1$ over the affine line 
whose fiber $\cX_t \subset \bP(1^4, 2, 1)$ over $t \in \bA^1$ is given by 
\begin{equation*}
f_4(x_0, \dots, x_4) + tg_3(x_0, \dots, x_4)x_5 = 0. 
\end{equation*}
Note that $X = \cX_0$. 

\begin{lemma}
\label{lemma-family}
There is a Zariski open neighborhood $U$ of $0 \in \bA^1$ such that: 
\begin{enumerate}
\item $\cX_t$ has terminal singularities for all $t \in U$. 
\item $\cX_t$ is rational for $t \in U \setminus \{ 0 \}$. 
\end{enumerate}
\end{lemma}

\begin{proof}
The fiber $\cX_0$ has terminal singularities by Lemma~\ref{lemma-X}. 
Since this condition is Zariski open in families~\cite[Corollary VI.5.3]{nakayama}, 
there is a Zariski open neighborhood $U$ of $0 \in \bA^1$ such that 
all fibers of $\cX_U \to U$ are terminal. 
Further, observe that for $t \neq 0$, projection away from the $x_5$-coordinate 
gives a birational map from  $\cX_t$ to $\bP(1^4,2)$. 
Indeed, this map is an isomorphism over the locus where $g_3(x_0, \dots x_4) \neq 0$ 
in $\bP(1^4,2)$. Hence $\cX_t$ is rational for $t \neq 0$. 
\end{proof}

Now we can prove Theorem~\ref{main-theorem}. 
By~\cite[Theorem~1.1]{voisin}, a very general quartic double solid is stably irrational. 
Taking such a $Y$ in the above construction and combining 
Lemmas~\ref{lemma-X} and~\ref{lemma-family}, we conclude that 
$\cX_U \to U$ is a family of fourfolds satisfying all of the required conditions. \qed


\providecommand{\bysame}{\leavevmode\hbox to3em{\hrulefill}\thinspace}
\providecommand{\MR}{\relax\ifhmode\unskip\space\fi MR }
\providecommand{\MRhref}[2]{%
  \href{http://www.ams.org/mathscinet-getitem?mr=#1}{#2}
}
\providecommand{\href}[2]{#2}


\begin{thebibliography}{1}

\bibitem{defernex}
Tommaso de~Fernex and Davide Fusi, \emph{Rationality in families of
  threefolds}, Rend. Circ. Mat. Palermo (2) \textbf{62} (2013), no.~1,
  127--135.

\bibitem{hacon-mckernan}
Christopher~D. Hacon and James M\textsuperscript{c}Kernan, \emph{On
  {S}hokurov's rational connectedness conjecture}, Duke Math. J. \textbf{138}
  (2007), no.~1, 119--136.

\bibitem{HPT}
Brendan Hassett, Alena Pirutka, and Yuri Tschinkel, \emph{Stable rationality of
  quadric surface bundles over surfaces}, arXiv:1603.09262 (2016).

\bibitem{nakayama}
Noboru Nakayama, \emph{Zariski-decomposition and abundance}, MSJ Memoirs,
  vol.~14, Mathematical Society of Japan, Tokyo, 2004.

\bibitem{totaro-hypersurfaces}
Burt Totaro, \emph{Hypersurfaces that are not stably rational}, J. Amer. Math.
  Soc. \textbf{29} (2016), no.~3, 883--891.

\bibitem{totaro-specialization}
\bysame, \emph{Rationality does not specialise among terminal varieties}, Math.
  Proc. Cambridge Philos. Soc. \textbf{161} (2016), no.~1, 13--15.

\bibitem{voisin}
Claire Voisin, \emph{Unirational threefolds with no universal codimension {$2$}
  cycle}, Invent. Math. \textbf{201} (2015), no.~1, 207--237.

\end{thebibliography}
\end{document}